\documentclass[12pt]{article}
\usepackage{amsmath,amsthm,amssymb,amsfonts,mathrsfs}
\usepackage[hypertex, bookmarks, colorlinks=true, plainpages=false,
citecolor = blue, urlcolor = blue, filecolor = blue]{hyperref}
\newtheorem{theorem}{Theorem}[section]
\newtheorem{lemma}[theorem]{Lemma}
\newtheorem{proposition}[theorem]{Proposition}
\newtheorem{corollary}[theorem]{Corollary}

\newtheorem{remark}[theorem]{Remark}

=23cm\headheight=0cm\topskip=0cm\topmargin=0mm
\def\a{\bar {a}}\def\b{\bar {b}}\def\c{\bar {c}}
\def\x{\bar {x}}\def\y{\bar {y}}
\def\z{\bar {z}}

\def\Nn{\mathbb N}

\def\Rn{\mathbb R}

\def\0{\sf 0}

\begin{document}
%\title{}
%\author{Seyed-Mohammad Bagheri}
%\maketitle

\begin{center}
{\Large\sc Definability in affine continuous logic}
\bigskip

{\bf Seyed-Mohammad Bagheri}

\end{center}

\begin{center}
{\footnotesize {\footnotesize Department of Pure Mathematics,
Faculty of Mathematical Sciences,\\
Tarbiat Modares University, Tehran, Iran, P.O. Box 14115-134}\\
e-mail: bagheri@modares.ac.ir,\ smbagherii@gmail.com}
\end{center}

\begin{abstract}
I study definable sets in affine continuous logic. Let $T$ be an affine theory.
After giving some general results, it is proved that if $T$ has a first order model,
its extremal theory is a complete first order theory and first order definable sets are affinely definable.
In this case, the type spaces of $T$ are Bauer simplices and they coincide with the sets of Keisler measures of the extremal theory.
In contrast, if $T$ has a compact model, definable sets are exactly the end-sets of definable predicates.
As an example, it is proved in the theory of probability algebras that one dimensional definable sets are exactly the intervals $[a,b]$.
\end{abstract}

{\sc Keywords}: {\small definable set, extreme type, extremally saturated model}

{\small {\sc AMS subject classification:} 03C40, 03C66, 52A07}

\section{Introduction}
%Formal logics are usually characterized by their expressive powers.
Definability theorems are usually consequences of the compactness theorem,
a property which is lost in the expressive powers higher than first order.
There is a rich family of general definability theorems in first order logic
much of which has been extended to continuous logic (see \cite{BBHU}).
The objective of the present paper is to study similar definability concepts in the framework of affine
(or linear\footnote{This terminology has been used in \cite{bagheri3, Isomorphism}.}) continuous logic.
This is the fragment of continuous logic obtained by reducing logical connectives to addition
and scalar multiplication (the affine structure of the value space $\Rn$).
This fragment enjoys a different form of the compactness theorem so that nontrivial
finite or compact structures find proper elementary extensions in the new setting.
%In particular, compact structures have nontrivial affine theories with less number of definable sets than in the full framework.
%Also, algebraic closure is no longer meaningful (although, definable closure is still meaningful).
It is however surprising that for first order models affine elementary equivalence implies first order elementary equivalence.

%Although affine logic is a fragment, it is a logic in its own right.
%One reason this logic is interesting is that it fits with the ultramean construction (a generalization of the ultraproduct construction),
%so that an affine variant of Keisler-Shelah isomorphism theorem holds.
%Also, several model theoretic notions have appropriate affine forms \cite{bagheri3, Isomorphism}.
%In particular, instead of ordinary types, one deals with affine types which is a fundamentally distinct notion from the point of view of analysis.
%In full continuous logic, types correspond to characters (Banach algebra homomorphisms)
%while in the affine fragment they correspond to positive linear functionals.
%So, full types form compact sets while affine types form compact convex sets.
%This additional structure on type spaces provides new connections between model theory and functional analysis.
%In particular, notions from convexity theory intervene in model theory.
%In fact, the two kinds of types are intimately related.
%Roughly, the space of affine types is the convex closure of the space of full types and the space of
%full types is the extreme boundary of the space of affine types.

In this paper, we study definability in the framework of affine logic.
Some basic results were already obtained in \cite{bagheri3}.
Here, we continue it by rather putting focus on the type spaces.
Affine definability has a close relation with facial types, specially in the case where the theory has a compact model.
In this case, definable sets are exactly the end-sets of definable predicates and
principal types are exactly the exposed ones.
On the other hand, if the theory has a first order model, all first order definable sets
are affinely definable and all first order tuples have extreme types.
It is also interesting that Keisler measures of a first order model $M$ correspond to types
of the affine theory of $M$ and they are realized in non first order models of this theory.

In the first section of the paper, a brief review of the affine logic is given.
For technical reasons, uniform continuity is replaced with Lipschitz continuity.
This does not impose important restriction on continuous structures as
every uniformly continuous function is the uniform limit of Lipschitz functions.
We can later add symbols for arbitrary definable predicates (as uniform limits of formulas)
and treat as if they are ordinary formulas.
In the next sections, some general facts on types and definable predicates (resp. functions, sets) are proved.
In most results extremal saturation is needed. This is a weaker form of saturation encompassing compact models as examples.

\section{Preliminaries}
A {\em Lipschitz language} is a first order language $L$ with some additional data.
To each function symbol $F$ is assigned a Lipschitz constant $\lambda_F\geqslant0$
and to each relation symbol $R$ is assigned a Lipschitz constant $\lambda_R\geqslant0$.
Also, the equality symbol is replaced with metric symbol $d$.
An $L$-structure is a complete metric space $(M,d)$ equipped with:
for each constant symbol $c$ an element $c^M\in M$,
for each $n$-ary function symbol $F$ a $\lambda_F$-Lipschitz function
$F^M:M^n\rightarrow M$ and for each $n$-ary relation symbol $R$ a $\lambda_R$-Lipschitz
function $R^M:M^n\rightarrow[0,1]$. In particular, $d:M^2\rightarrow[0,1]$ is $1$-Lipschitz.
Here, $M^n$ is equipped with the metric $\sum_{i=1}^n d(x_i,y_i)$.\
\emph{Affine formulas} (formulas for short) are inductively defined as follows:
$$1, \ \ R(t_1,...,t_n), \ \ \ r\cdot\phi,\ \ \ \phi+\psi,\ \ \ \inf_x\phi,\ \ \ \sup_x\phi$$
where $R\in L$ is an $n$-ary relation symbol (or metric symbol $d$), $t_1,...,t_n$ are $L$-terms and $r\in\Rn$.
For each $\phi(\x)$ and $\a\in M$, $\phi^M(\a)\in\Rn$ is defined in the usual way.
One can then assign to each formula $\phi$ a Lipschitz constant
$\lambda_{\phi}$ and a bound $\bf{b}_{\phi}$ such that for every $M$
$$|\phi^M(\x)|\leqslant{\bf b}_\phi,\ \ \ \ \ |\phi^M(\x)-\phi^M(\y)|
\leqslant\lambda_{\phi} d(\x,\y)\ \ \ \ \ \ \ \forall\x\y\in M.$$

Affine continuous logic AL is based on the formulas defined above.
In full continuous logic CL one further allows $\phi\vee\psi$ and $\phi\wedge\psi$.
So, AL is a fragment of CL.

We now recall the ultramean construction shortly (see \cite{Isomorphism} for further detail).
Let $I$ be a nonempty index set.
An \emph{ultracharge} on $I$ is a finitely additive probability measure $\mu:P(I)\rightarrow[0,1]$.
Ultrafilters may be regarded as $\{0,1\}$-valued ultracharges.
Let $\mu$ be an ultracharge on $I$.
For each $i\in I$, let $(M_i,d_i)$ be an $L$-structure.
For $a,b\in \prod_i M_i$ define
$$d(a,b)=\int d_i(a_i,b_i) d\mu.$$
Then $d$ is a pseudometric and $d(a,b)=0$ is an equivalence relation.
The equivalence class of $(a_i)$ is denoted by $[a_i]$,
the quotient set by $M=\prod_{\mu} M_i$ and the induced metric on
$M$ again by $d$.
We define an $L$-structure on $M$ as follows.
For $c, F, R \in L$ (say unary for simplicity) set
$$c^M=[c^{M_i}]$$
$$F^M([a_i])=[F^{M_i}(a_i)]$$
$$R^M([a_i])=\int R^{M_i}(a_i)d\mu.$$

These integrals coincide with the corresponding ultralimits if $\mu$ is an ultrafilter.
In this case, one obtains the usual ultraproduct.
The structure $M$ defined above may not be complete. It can however be completed
by usual arguments and the completion satisfies the same conditions as $M$.
The following result is an affine variant of the ultraproarduct theorem.

\begin{theorem} \emph{(Ultramean theorem)}
For every affine formula $\phi(x_1,...x_n)$ and $[a^1_i],...,[a^n_i]$
$$\phi^M([a^1_i],...,[a^n_i])=\int\phi^{M_i}(a^1_i,...,a^n_i) d\mu.$$
\end{theorem}

If $M_i=N$ for all $i$, one obtains the \emph{powermean} $N^{\mu}$.
The diagonal map $a\mapsto[a]$ is then elementary (preserves all formulas).

Expressions of the form $\phi\leqslant\psi$ are called \emph{conditions}.
$\phi=\psi$ abbreviates the set $\{\phi\leqslant\psi,\psi\leqslant\phi\}$.
A theory is a set of closed (i.e. without free variables) conditions.
A theory $T$ is \emph{affinely satisfiable} if for every
$\phi_1\leqslant\psi_1,\ ...,\ \phi_n\leqslant\psi_n$ in $T$ and $r_1,...,r_n\geqslant0$,
the condition $\sum_ir_i\phi_i\leqslant\sum_ir_i\psi_i$ is satisfiable.
The set of all such combinations is called the \emph{affine closure} of $T$.
For sentences $\phi,\psi$,\ \ $T\vDash\phi\leqslant\psi$ is defined as usual.
$T$ is complete if for each sentence $\phi$ there is a unique $r$ such that
$T\vDash\phi=r$. A consequence of the ultramean theorem is the affine compactness theorem.

\begin{theorem} \label{compactness} (see \cite{bagheri3, Isomorphism})
Every affinely satisfiable theory is satisfiable.
\end{theorem}

In applications, one may use a somewhat weaker condition.
It is sufficient to prove that every condition in the affine closure of $T$ is approximately satisfiable.
Affine forms of elementary embedding $\preccurlyeq$ and elementary equivalence
$\equiv$ are defined in the usual way.
For example, $M\equiv N$ if $\sigma^M=\sigma^N$ for every sentence $\sigma$.
One proves by affine compactness that every model of cardinality at least two
has arbitrarily large elementary extensions.
This can be also proved by taking powermeans over suitable probability measures.
So, AL is a proper fragment of CL, i.e. has a strictly weaker expressive power.
However, most basic CL (or first order) theorems have appropriate counterparts in AL.
In particular, a Keisler-Shelah type isomorphism theorem holds in AL.
More details can be found in \cite{bagheri3, Isomorphism}.

\section{Types} \label{section2}
Let $T$ be a complete theory in $L$.
The set of $T$-equivalence classes of formulas with variables $\x$ forms a
normed vector space with $$\|\phi\|=\sup_{\a\in M}|\phi^M(\a)|$$
where $M$ is any model of $T$.
This space is denoted by $\mathbb D_n(T)$ or even $\mathbb D_n$.
We usually identify $\phi$ with its $T$-equivalence class or its interpretation $\phi^M$.
So, $\mathbb D_n\subseteq{\bf C}_b(M^n)$ is a partially ordered normed space.

A \emph{partial type} is a set of conditions $\phi(\x)\leqslant\psi(\x)$
which is satisfiable in some model of $T$.
Maximal partial types over $\x$, where $|\x|=n$, are called $n$-types.
So, if $p(\x)$ is a type, for each $\phi(\x)$ there is a unique $r$ such that $\phi=r\in p$.
This $r$ is denoted by $p(\phi)$. Then $\phi\mapsto p(\phi)$ defines a
positive linear functional on $\mathbb D_n$ with $p(1)=1$ (hence $\|p\|=1$).
Conversely, by affine compactness, every positive linear functional
$p:\mathbb D_n\rightarrow\Rn$ with $p(1)=1$ defines an $n$-type.
We usually identify each type with the corresponding linear functional.
For $\a\in M$, the type $\phi\mapsto\phi^M(\a)$ is denoted by $tp(\a)$.
A type $p$ is realized in $M$ if $p=tp(\a)$ for some $\a\in M$.

The set of $n$-types is denoted by $S_n(T)$.
The notion of a type over a set $A\subseteq M$ is defined in the usual way.
$S_n(A)$ denotes the set of $n$-types over $A$.
A model $M$ is $\kappa$-\emph{saturated} if for each $A\subseteq M$ with $|A|<\kappa$,
every $p\in S_n(A)$ is realized in $M$.
It is \emph{strongly $\kappa$-homogeneous} if for each tuples $\a,\b$ of length
$\lambda<\kappa$ with $(M,\a)\equiv (M,\b)$ there is an automorphism $f$ such that $f(\a)=\b$.
Every model has $\kappa$-saturated elementary extensions for every $\kappa\geqslant\aleph_0$.
This is a consequence of the following lemma and AL variant of the elementary chains theorem.

The elementary diagram of a model $M$ is defined in the usual way, i.e. $$ediag(M)=\{0\leqslant\phi(\a):\ 0\leqslant\phi^M(\a)\}.$$
This is a theory in the language $L_M$ in which there is a constant symbol for every $a\in M$.
It is then proved (assuming $M\subseteq N$) that $M\preccurlyeq N$ if and only if $N\vDash \mbox{ediag}(M)$.

\begin{lemma}
If $A\subseteq M$, every $p\in S_n(A)$ is realized in an elementary extension of $M$.
\end{lemma}
\begin{proof}
We have to show that $\Sigma=ediag(M)\cup p(\x)$ is affinely satisfiable.
Note that $ediag(M)$ is affinely closed (i.e. it coincides with its affine closure).
Let $0\leqslant\phi^M(\a,\b)$ where $\a\in A$ and $\b\in M-A$. Let $N$ be a model of $p(\x)\cup Th(M,a)_{a\in A}$.
Then, $N$ is a model of $0\leqslant\sup_{\y}\phi(\a,\y)$ as well as $p(\x)$.
\end{proof}

Similar arguments show that every model has a strongly $\kappa$-homogeneous elementary extension.
The type space $S_n(T)$ may be equipped with various topologies. The \emph{logic topology}
is generated by the sets of the form $$[r<\phi(\x)]=\{p\in S_n(T):\ \ \ r<p(\phi)\}.$$
Equivalently, writing $\hat\phi(p)=p(\phi)$, the logic topology is
the coarsest topology in which every $\hat\phi:S_n(T)\rightarrow\Rn$ is continuous.
A basis for this topology is the family of sets $\bigcap_{i=1}^n[0<\phi_i(\x)]$.
The logic topology is compact by the Banach-Alaoglu theorem. $S_n(T)$ is also convex.

A subset $A$ of a vector space $V$ is said to be convex if for each
$x,y\in A$ and $0\leqslant\gamma\leqslant1$, one has that $\gamma x+(1-\gamma)y\in A$.
Clearly, if $p,q$ are $n$-types then so is $\gamma p+(1-\gamma)q$.
So, $S_n(T)$ is a weak* compact convex subset of $\mathbb{D}_n^*$.
It is also worth noting that the weak* dual of $\mathbb D_n^*$ coincides with $\mathbb D_n$ itself (see \cite{Conway} p.125).
Then, an application of the Hahn-Banach separation theorem (\cite{Conway} p.111) shows that

\begin{proposition} \label{closed convex}
Closed convex subsets of $S_n(T)$ are exactly the sets defined by partial types, i.e.
sets of the form $\{p\in S_n(T):\ \ \Gamma\subseteq p\}$ where $\Gamma$ is a set of conditions.
\end{proposition}

Let $V$ be a tvs and $K\subseteq V$ be compact convex.
A closed $F\subseteq K$ is called a \emph{face} if for each $p,q\in K$ and $0<\gamma<1$,\ \
$\gamma p+(1-\gamma)q\in F$ implies that $p,q\in F$. One point faces are called \emph{extreme points}.
By the Krein-Milmann theorem, $K$ is equal to the closure of the convex hull of the set of extreme points.

\begin{theorem} \emph{(\cite{Simon} Th. 8.3)}
Let $K\subseteq V$ be compact convex and $f\in V^*$. If $f|_K$ is not constant and
$r=\sup_{p\in K}f(p)$, then $F=f^{-1}(r)$ is a proper face of $K$.
\end{theorem}

Faces of the above form are called \emph{exposed faces}.
Also, $F$ is said to be exposed by $f$. A point $p$ is exposed if $\{p\}$ is so.

Let $\Gamma(\x)$ be a set of conditions satisfiable with $T$ (a partial type) where $|\x|=n$.
$\Gamma$ is called a \emph{facial type} if the set $$\{p\in S_n(T):\ \Gamma\subseteq p\}$$ is a face of $S_n(T)$.
The following remark is a consequence of general facts for compact convex sets (see \cite{Simon}).

\begin{remark} \label{facial}
\emph{If $\Gamma_i(\x)$ is facial for each $i\in I$, then so is $\bigcup_{i\in I}\Gamma_i$ (if satisfiable).
If $\Gamma(\x)$ is facial and $\Gamma(\x)\vDash\theta(\x)\leqslant0$,
then $\Gamma(\x)\cup\{0\leqslant\theta(\x)\}$ is facial (if satisfiable).}
\end{remark}

A structure $M$ \emph{extremally $\kappa$-saturated} if for every $A\subseteq M$
with $|A|<\kappa$, every extreme type in $S_n(A)$ is realized in $M$.
This notion is weaker than being $\kappa$-saturated since only extreme types are intended.
Since every face contains an extreme type, an extremally $\kappa$-saturated model
realizes every facial type with less that $\kappa$ parameters.

%\begin{proposition}\label{sat1} (see \cite{Isomorphism} or \cite{extreme})
%Let $\mathcal F$ be a countably incomplete ultrafilter on $I$ and for each $i\in I$, $M_i$ be an $L$-structure where $L$ is countable.
%Then, $M=\prod_{\mathcal F}M_i$ is extremally $\aleph_1$-saturated.
%\end{proposition}

Let $\mathcal F$ be an ultrafilter on a nonempty set $I$ and $\kappa,\lambda$ be infinite cardinals with $\lambda<\kappa$.
A function $f:S_\omega(\lambda)\rightarrow\mathcal F$ is monotonic if $f(\tau)\supseteq f(\eta)$
whenever $\tau\subseteq\eta$. It is additive if $f(\tau\cup\eta)= f(\tau)\cap f(\eta)$.\ \
$\mathcal F$ is $\kappa$-good if for every $\lambda<\kappa$ and every monotonic $f:S_\omega(\lambda)\rightarrow\mathcal F$
there exists an additive $g$ such that $g(\tau)\subseteq f(\tau)$ whenever $\tau\in S_\omega(\lambda)$.

\begin{proposition} \label{sat1}
Let $\kappa$ be an infinite cardinal and $\mathcal F$ be a countably incomplete
$\kappa$-good ultrafilter on a set $I$. Let $|L|+\aleph_0<\kappa$ and for each $i\in I$,
$M_i$ be an $L$-structure. Then, $M=\prod_{\mathcal F}M_i$ is extremally $\kappa$-saturated.
\end{proposition}
\begin{proof}
The proof is an adaptation of the proof of Theorem 6.1.8 of \cite{CK1} for the present situation.
For any set $A\subseteq M$ with $|A|<\kappa$ one has that
$(M,a)_{a\in A}\simeq\prod_{\mathcal F}(M_i,a_i)_{a\in A}$.
So, we may forget the parameters and prove that every extreme type of $Th(M)$ is realized in $M$.
For simplicity assume $|\x|=1$ and let $p(x)$ be extreme.
Let $\mathbb U(M)$ be the set of ultracharges on $M$. This is a compact convex set whose extreme
points are ultrafilters. Let
$$V=\big\{\wp\in\mathbb{U}(M):\ \ \ p(\phi)=\int\phi^M(x)d\wp\ \ \ \ \forall\phi\big\}.$$
The type $p$ induces a positive linear functional on the space of functions $\phi^M(x)$.
By the Kantorovich extension theorem (\cite{Aliprantis-Inf}, Th. 8.32)
it extends to a positive linear functional $\bar p$ on $\ell^{\infty}(M)$.
Then, $\bar p$ is represented by integration over an ultracharge on $M$
so that $V$ is non-empty.
Moreover, $V$ is a closed face of $\mathbb{U}(M)$.
In particular, assume for $r\in(0,1)$ one has that $r\mu+(1-r)\nu=\wp\in V$.
Define the types $p_\mu$, $p_\nu$ by setting for each $\phi(x)$
$$p_\mu(\phi)=\int\phi^M d\mu,\ \ \ \ \ \ \ \ p_\nu(\phi)=\int\phi^M d\nu.$$
Then, $rp_\mu+(1-r)p_\nu=p$.
We have therefore that $p_\mu=p_\nu=p$ and hence $\mu,\nu\in V$.

Let $\wp$ be an extreme point of $V$.
Then, $\wp$ is an extreme point of $\mathbb{U}(M)$ and hence it corresponds to
an ultrafilter, say $\mathcal D$ (not to be confused with the ultrafilter $\mathcal F$ on $I$).
We have therefore that $$\ \ \ \ p(\phi)=\int_M\phi^M(x)d\wp=
\lim_{\mathcal D,x}\phi^M(x) \ \ \ \ \ \ \ \forall\phi.$$
Since $|L|+\aleph_0<\kappa$, we may assume $p$ is axiomatized by a family of conditions
$$\{0\leqslant\phi(x):\ \ \phi\in\Sigma\}\equiv p(x)$$ where $\Sigma$ is a set of formulas with $|\Sigma|<\kappa$.
For this purpose, one may use formulas with rational coefficients.
Let $$\Sigma^+=\{\phi+r:\ \ \phi\in\Sigma,\ r>0\ \mbox{is\ rational}\}.$$
Let $I_1\supseteq I_2\supseteq\cdots$ be a chain such that
$I_n\in\mathcal F$ and $\bigcap_n I_n=\emptyset$.
Let $f:S_\omega(\Sigma^+)\rightarrow\mathcal F$ be defined as follows.
$f(\emptyset)=I$ and for nonempty $\tau\in S_\omega(\Sigma^+)$
$$f(\tau)=I_{|\tau|}\cap\big\{i\in I:\ \ \ 0<\sup_x\bigwedge_{\phi\in\tau}\phi^{M_i}(x)\big\}.\ \ \ \ \ \ \ \ \ (*)$$
Since $\mathcal D$ is an ultrafilter, there exists $a\in M$ such that
$0<\phi^M(a)$ for every $\phi\in\tau$.
We have therefore that $f(\tau)\in\mathcal F$.
Also, $f(\tau)\supseteq f(\eta)$ whenever $\tau\subseteq\eta$.
Since $\mathcal F$ is $\kappa$-good, there exists an additive function $g:S_\omega(\Sigma^+)\rightarrow\mathcal F$
such that $g(\tau)\subseteq f(\tau)$ for every $\tau\in\Sigma^+$.
Let $$\tau(i)=\{\phi\in\Sigma^+:\ i\in g\{\phi\}\}.$$
If $\phi_1,...,\phi_n\in\tau(i)$ are distinct, then
$$i\in g\{\phi_1\}\cap\cdots\cap g\{\phi_n\}=g\{\phi_1,...,\phi_n\}\subseteq f\{\phi_1,...,\phi_n\}\subseteq I_n.$$
In particular, if $|\tau(i)|\geqslant n$ then $i\in I_n$
and hence $\tau(i)$ is finite for each $i$ as $\bigcap_n I_n=\emptyset$.
We have also that
$$i\in\bigcap\big\{g\{\phi\}:\ \phi\in\tau(i)\big\}=g(\tau(i))\subseteq f(\tau(i))\subseteq I_{|\tau(i)|}.$$

Now, we define $a\in M$ which realizes $p(x)$. By $(*)$, we may choose $a_i\in M_i$ such that
$$0<\bigwedge_{\phi\in\tau(i)}\phi^{M_i}(a_i).$$
Fix $\phi\in\Sigma^+$. For each $i\in g\{\phi\}\in\mathcal F$ one has that
$\phi\in\tau(i)$ and so $0<\phi^{M_i}(a_i)$. This shows that $0\leqslant\phi^M(a)$.
We conclude that $0\leqslant\phi^M(a)$ for every $\phi\in\Sigma$ and hence $a$ realizes $p(x)$.
\end{proof}

In particular, every compact model is extremally $\kappa$-saturated for every $\kappa$. %Below, we denote Ext$(S_n(T))$ by $E_n(T)$.
The second natural topology on $S_n(T)$ is the metric topology.
Let $M$ be an $\aleph_0$-saturated model of $T$.
Then $$\mathbf{d}(p,q)=\inf\{d(\a,\b):\ \a,\b\in M,\ \a\vDash p,\ \b\vDash q\}$$
defines a metric whose topology is finer that the logic one.
The metric topology is used in the study of definable predicates.

\section{Definable predicates}
As before, $T$ is a complete theory in $L$. Unless otherwise stated, by definable we mean without parameters.
We assume all parameters needed to define a notion are already named in the language.
A predicate $P:M^n\rightarrow\Rn$ is \emph{definable} if there is a sequence $\phi_k(\x)$
of formulas such that $\phi_k^M\rightarrow P$ uniformly on $M^n$.
This sequence determines a definable predicate on every $N\vDash T$ which we denote by $P^N$.
We can treat definable predicates as interpretations of new relation symbols added to the language.
Although, they are to be interpreted by uniformly continuous functions (with predetermined moduli of continuity).
The following is then routine.% (see also \cite{bagheri3}).

\begin{proposition} \label{definitional expansion}
Let $P:M^n\rightarrow\Rn$ be definable. If $N\preccurlyeq M$ then $P^N=P|_N$ and
$(N,P^N)\preccurlyeq(M,P)$. If $M\preccurlyeq N$, then $(M,P)\preccurlyeq(N,P^N)$.
Also, if $M=\prod_\mu M_i$, then $P^M(\a)=\int P^{M_i}(\a_i)d\mu$ for every $\a\in M$.
\end{proposition}

For a formula $\phi$ in $L\cup\{P\}$, the notion of positive (resp. negative) occurrence of
$P$ in $\phi$ is defined as in \cite{CK1}, i.e. a positive (resp. negative) occurrence of $P$
is one which is in the scope of an even (resp. odd) number of negative real coefficients.

\begin{proposition}
Let $P_1,...,P_m$ be definable predicates on $M\vDash T$ and $\bar L=L\cup\{P_1,...,P_m\}$.
Let $\Sigma$ be an affinely closed set of condition in $\bar L$.
If every condition in $\Sigma$ is satisfied in a model of $T$, then $\Sigma$ is satisfied in a model of $T$.
\end{proposition}
\begin{proof}
For simplicity let $m=1$. For each $k$ assume $\|P^M-\phi_k^M\|\leqslant\frac{1}{k}$.
Assume also that every condition in $\Sigma$ is of the form $0\leqslant\theta$.
For $0\leqslant\theta$ in $\Sigma$, let $\theta_k$ be the $L$-formula
obtained by replacing every positive occurrence of $P$ with $\phi_k+\frac{1}{k}$ and every
negative occurrence of $P$ with $\phi_k-\frac{1}{k}$.
Let $\Sigma_k$ be set of conditions of the form $0\leqslant\theta_k$ for any $0\leqslant\theta$ in $\Sigma$.
Then $T\cup(\cup_k\Sigma_k)$ is affinely satisfiable (in $M$) and its total models satisfy $T\cup\Sigma$.
\end{proof}

Let $K\subseteq V$ be convex where $V$ is a tvs.
A function $f:K\rightarrow\Rn$ is \emph{convex} if for every $p,q\in K$
and $0\leqslant\gamma\leqslant 1$, one has that
$$f(\gamma p+(1-\gamma)q)\leqslant\gamma f(p)+(1-\gamma)f(q).$$
If equality holds (i.e. $f$ and $-f$ are convex), $f$ is called \emph{affine}.

Given a formula $\phi(\x)$, the function defined on $S_n(T)$ by $\hat\phi(p)=p(\phi)$ is affine,
logic-continuous and $\lambda_{\phi}$-Lipschitz.
Clearly, $\hat\phi=\hat\psi$ if and only if $\phi$ and $\psi$ are $T$-equivalent.
Also, for each $M\vDash T$, a sequence $\phi^M_k$ is Cauchy if and only if $\hat\phi_k$ is Cauchy.
%However, the uniform limit of $\phi^M_k$ need not be Lipschitz.
The set of affine logic-continuous functions on $S_n(T)$ is denoted by $\mathbf{A}(S_n(T))$.
This is a Banach space with the supremum norm (see \cite{Alfsen}).

\begin{proposition} \label{Lipschitz}
The following are equivalent for every $\xi: S_n(T)\rightarrow\Rn$:

  (i) $\xi\in\mathbf{A}(S_n(T))$

  (ii) There is a sequence $\phi_k$ of formulas such that $\hat\phi_k$ converges to $\xi$ uniformly.
\end{proposition}
\begin{proof} $(i)\Rightarrow(ii)$: Every $\hat\phi$ is clearly affine and
logic-continuous. Moreover, the subspace of $\mathbf A(S_n(T))$
consisting of these functions contains constant maps and separates points.
Hence, it is dense in $\mathbf A(S_n(T))$ (\cite{Alfsen} Cor. I.1.5).
The inverse direction is obvious.
% $(ii)\Rightarrow(i)$: If $\hat\phi_k$ converges uniformly to $\xi$ then $\xi$ is affine and logic-continuous.
\end{proof}

The set of definable predicates (on $M$ or on any other model of $T$) is denoted by
$\mathbf D_n(T)$ or even $\mathbf D_n$. This is the completion of $\mathbb D$.
Note that if $M$ realizes all types, then
$$\sup_{\a\in M}|\phi^M(\a)|=\sup_{p\in S_n(T)}|\hat\phi(p)|.$$
In particular, $\|\phi\|=\|\hat\phi\|$.
We deduce by Proposition \ref{Lipschitz} that $\mathbf{D}_n$ and $\mathbf{A}(S_n(T))$ are isometrically isomorphic.

The \emph{epigraph} of a function $f:X\rightarrow\Rn$ is the set
$$\mbox{epi}(f)=\{(x,r)\ : \ f(x)\leqslant r\}.$$
We say $P:M^n\rightarrow\Rn$ has a \emph{type-definable epigraph}
if there is a set $\Phi(\x)$ of formulas such that
$$\mbox{epi}(P)=\{(\a,r)\ : \  \phi^M(\a)\leqslant r\ \ \mbox{for\ every}\ \phi\in\Phi\}.$$
It is known that a function $f:K\rightarrow\Rn$ on a convex subset of a vector space is
convex if and only if its epigraph is a convex set (\cite{Aliprantis-Inf} Lem. 5.39).

\begin{proposition} \label{dfncriterion2}
Let $M$ be $\aleph_0$-saturated. Then, a predicate $P:M^n\rightarrow\Rn$ is
definable if and only if both $P$ and $-P$ have type-definable epigraphs.
\end{proposition}
\begin{proof} Assume $P$ is definable.
Take a sequence $\phi_k$ of formulas such that $\|P-\phi_k\|\leqslant\frac{1}{k}$ for all $k$.
Then, $$\mbox{epi}(P)=\big\{(\a,r)\ \big|\ \phi_k^M(\a)-\frac{1}{k}\leqslant r\ \ \ \ \ \forall k<\omega\big\}.$$
Similarly, the epigraph of $-P$ is type-definable.
Conversely assume the epigraphs of $P$ and $-P$ are type-definable.
Define a map $\xi: S_n(T)\rightarrow\Rn$ by $\xi(p)=P(\a)$ where $\a\vDash p$.
It is clear that $\xi$ is well-defined and logic-continuous.
We show that it is affine. Assume $$\mbox{epi}(P)=\{(\a,r)\ :\ \forall\phi\in\Phi,\ \ \phi^M(\a)\leqslant r\}.$$
Then, for every $p$ and $\phi\in\Phi$ one has that $p(\phi)\leqslant\xi(p)$.
Fix $p_1$, $p_2$ and let $\c$ realizes $\gamma p_1+(1-\gamma)p_2$.
Then, for all $\phi\in\Phi$ one has that
$$\phi^M(\c)=\gamma p_1(\phi)+(1-\gamma)p_2(\phi)\leqslant\gamma\xi(p_1)+(1-\gamma)\xi(p_2).$$
Therefore, by the assumption
$$\xi(\gamma p_1+(1-\gamma)p_2)=P(\c)\leqslant\gamma\xi(p_1)+(1-\gamma)\xi(p_2)$$
which shows that $\xi$ is convex. Similarly, $-\xi$ is convex.
We conclude that $\xi$ is affine. By Proposition \ref{Lipschitz},
$\hat\phi_k\stackrel{u}\rightarrow\xi$ for some sequence $\phi_k$.
Therefore, $\phi^M_k\stackrel{u}\rightarrow P$.
\end{proof}

Among the standard definability theorems which have appropriate affine variants are the Svenonius definability theorem,
Beth's definability theorem and the following result whose proofs can all be found in \cite{bagheri3}.

\begin{proposition} \label{aut}
Let $M$ be strongly $\aleph_1$-homogeneous and $P$ be definable with parameters.
Then $P$ is $\emptyset$-definable if and only if it is preserved by every automorphism of $M$.
\end{proposition}

\section{Definable functions}
A function $f:M^m\rightarrow M^n$ is \emph{definable} if $d(f(\x),\y)$ is definable where $|\y|=n$.
The following equalities shows that $f$ is definable if and only if its graph $G_f$ is definable:.
$$d((\x,\y),G_f)=\inf_{\bar{u}}[d(\x,\bar{u})+d(f(\bar{u}),\y)]$$
$$d(f(\x),\y)=\inf_{\bar{v}}[d(\x\bar{v},G_f)+d(\bar{v},\y)].$$
In particular, if $f$ is definable and invertible, then $f^{-1}$ is definable.
Also, one verifies that $f$ is $\lambda$-Lipschitz if and only if $d(f(\x),\y)$ is so.

\begin{lemma} \label{P-D}
If $f$ is definable, then for each definable $P(u,\y)$,\ \ $P(f(\x),\y)$ is definable.
\end{lemma}
\begin{proof}
First assume $P$ is the formula $\phi$ and show that
$$\phi^M(f(\x),\z)\leqslant\inf_{\y}[\phi^M(\y,\z)+\lambda_{\phi} d(f(\x),\y)]\leqslant\phi^M(f(\x),\z).$$
Then, assume $\phi_k^M\stackrel{u}{\rightarrow}P$ and deduce that $\phi_k^M(f(\x),\y)\stackrel{u}{\rightarrow}P(f(\x),y)$.
\end{proof}

As a consequence, if $f$ and $g$ are definable, then so is $g\circ f$.
Also, since projections are definable, $f=(f_1,...,f_n)$ is definable if and only if $f_1,...,f_n$ are so.
We have also the following stronger result which holds if $M$ is $\aleph_0$-saturated.
A set $X\subseteq M^n$ is \emph{type-definable} if it is the set of common solutions of
a family of conditions.

\begin{proposition} \label{graph}
Let $M$ be $\aleph_0$-saturated. Then, $f:M^n\rightarrow M$ is definable if and only if $G_f$ is type-definable.
\end{proposition}
\begin{proof} For the nontrivial direction, assume $\Gamma(\x,u)$ is a set of conditions
of the form $\phi(\x,u)\leqslant0$ which type-defines $G_f$. Let
$$\Lambda_r(\x,y)=\big\{\inf_u[\alpha\phi(\x,u)+d(u,y)]\leqslant r\ :\ \ \
\phi(\x,u)\leqslant0\in\Gamma,\ \alpha\geqslant0\big\}.$$
Clearly, if $d(f(\a),b)\leqslant r$, then $(\a,b)$ satisfies $\Lambda_r(\x,y)$ (set $u=f(\a)$).
Conversely, if $(\a,b)$ satisfies $\Lambda_r(\x,y)$, then the type
$$\{\phi(\a,u)\leqslant 0\ : \ \phi(\x,u)\leqslant0\in\Gamma\}\cup\{d(u,b)\leqslant r\}$$
is affinely satisfied in $M$. So, by saturation, it is satisfied by some $c\in M$.
Then $f(\a)=c$ and $d(f(\a),b)\leqslant r$. We therefore have that
$$d(f(\a),b)\leqslant r\ \ \Leftrightarrow\ \ (\a,b)\vDash\Lambda_r(\x,y)\ \ \ \ \ \ \ \forall\a,b$$
and hence the epigraph of $d(f(\x),y)$ is type-definable.
Similarly, one shows that the epigraph of $-d(f(\x),y)$ is type-definable.
We conclude by Proposition \ref{dfncriterion2} that $d(f(\x),y)$ is a definable predicate.
\end{proof}

The following is a definable variant of the existence of invariant probability measures
for continuous functions on compact metric spaces.

\begin{proposition}
Let $f:M\rightarrow M$ be definable. Then there exists a type $p(x)\in S_1(T)$
which is $f$-invariant, i.e. $p(\phi(x))=p(\phi(f(x))$ for every $\phi(x)$.
In particular, if $M$ is $\aleph_0$-saturated, then there exists $c\in M$ such that $c\equiv f(c)$.
\end{proposition}
\begin{proof} $f$ induces an affine continuous map
$\hat f:S_1(T)\rightarrow S_1(T)$ by $$\hat f(p)(\phi(x))=p(\phi(f(x))).$$
By the Schauder-Tychonoff fixed point theorem (\cite{Conway} p.150),
$\hat f$ has a fixed point $p$. Then, $p(\phi(x))=p(\phi(f(x))$ for each $\phi(x)$.
If $M$ is $\aleph_0$-saturated, let $c\in M$ realize $p$. Then $c\equiv f(c)$.
\end{proof}

%For example, if $M$ is an $\aleph_0$-saturated group, for each $a$ there exists $c$ such that $ac\equiv c$ in the language of $(G,a)$.
We may also consider an Abelian group $G$ of definable bijections $f:M\rightarrow M$ and use
the Markov-Kakutani theorem (\cite{Conway} p.151) to find $p\in S_1(T)$ such that $\hat f(p)=p$ for every $f\in G$.
%In particular, if $M$ is an $\aleph_1$-saturated group and $G\subseteq M$ is a countable Abelian subgroup,
%then there exists $c\in M$ such that for every $g\in G$,\ \ $gc\equiv c$ in the language of $(M,a)_{a\in G}$.
%In particular, if $G$ is an Abelian (discrete) subgroup with $|A|<\kappa$,
%then there exists $p(x)\in S_1(G)$ such that $p(x)=p(gx)$ for every $g\in G$.

\section{Definable sets}
A closed $D\subseteq M^n$ is definable if $d(\x,D)=\inf_{\a\in D}d(\x,\a)$ is definable.
We use the convention $\inf_{\a\in\emptyset}P(\a)=\|P\|$.
Definable sets are not closed under Boolean combinations. However, if $D,E$ are definable, then so are
$D\times E$ and $\{\x:\ \exists y\ \x y\in D\}$.
%If $D\subseteq M^{n+1}$ is definable, its projection $\pi(D)$ on $M^n$ is definable as $d(\x y,D)=\inf_y d(\x,\pi(D))$.

\begin{remark} \label{3conditions}
Let $D\subseteq M^n$ be definable and set $P(\x)=d(\x,D)$.
Then the following properties hold for every $\x,\y\in M^n$:

(i) $0\leqslant P(\x)$

(ii) $P(\x)-P(\y)\leqslant d(\x,\y)$

(iii) $0\leqslant\inf_{\x}\sup_{\y}[sP(\x)-rP(\y)-sd(\x,\y)]$
\ \ \ \ \ \ \ \ $\forall r,s\geqslant0$.
\end{remark}

The inequality (iii) states that for each $\a$,
$\{P(\y)\leqslant0,\ d(\a,\y)\leqslant P(\a)\}$ is affinely approximately satisfiable in $M$.
The properties (i)-(iii) characterize definable sets.

\begin{proposition} \label{dfnconditions}
Let $M$ be extremally $\aleph_0$-saturated and $P:M^n\rightarrow\Rn$ be definable.
If $P$ satisfies (i)-(iii) above, then $P(\x)=d(\x,D)$ where $D=Z(P)=\{\a:\ P(\a)=0\}\neq\emptyset$.
\end{proposition}
\begin{proof} As stated in the previous section, $P$ can be regarded as a formula so that $(M,P)$ is extremally saturated.
Then, taking $s=0$, $r=1$ in (iii) and using extremal saturation,
one checks that $D$ is nonempty. By (ii), we have that $P(\x)\leqslant d(\x,\y)$
for all $\y\in D$. Hence $P(\x)\leqslant d(\x,D)$. For the inverse inequality, fix $\a\in M$.
By (iii) (and using extremal saturation),
$$\{P(\y)\leqslant 0,\ d(\a,\y)\leqslant P(\a)\}$$ is affinely satisfiable in $M$.
By Remark \ref{facial}, this is a facial type since $P(\y)\leqslant0$ implies that $\y\in D$.
So, by (i) and extremal saturation, there exists $\b$ such that
$$P(\b)=0, \ \ d(\a,\b)\leqslant P(\a).$$
Therefore, $d(\a,D)\leqslant d(\a,\b)\leqslant P(\a)$
and hence $d(\x,D)\leqslant P(\x)$ for all $\x\in M$.
\end{proof}

\begin{proposition} \label{dfnrestriction}
Let $M$ be extremally $\aleph_0$-saturated and $M\preccurlyeq N$.
If $D\subseteq N^n$ is definable, then $C=D\cap M^n$ is definable and for each
$\x\in M$, $d(\x,D)=d(\x,C)$. In particular, $(M,d(\x,C)\preccurlyeq(N,d(\x,D))$.
If $D\neq\emptyset$ then $C\neq\emptyset$.
\end{proposition}
\begin{proof} By Proposition \ref{definitional expansion}, $Q(\x)=d(\x, D)|_{M^n}$ is definable in $M$
and $(M,Q)\preccurlyeq (N,d(\x,D))$.
Note that $Q$ satisfies conditions (i)-(iii) in Remark \ref{3conditions}.
So, since the zeroset of $Q$ is $C$, by Proposition \ref{dfnconditions} we have that $Q(\x)=d(\x,C)$.
For the last part, use the fact that $\inf_xd(x,D)<1$.
\end{proof}

Similarly, if $M\preccurlyeq N$ and $N$ is extremally $\aleph_0$-saturated, one promotes a definable
$C\subseteq M^n$ to a definable $D\subseteq N^n$ such that $C=D\cap M^n$.

Assume $D\subseteq M^n$ and $P(\x,\y)\leqslant P(\x,\z)+\lambda d(\z,\y)$ for all $\x,\y,\z$.
Take the infimum first over $\y\in D$ and then over $\z\in M^n$ to obtain
$$\inf_{\y\in D} P(\x,\y)\leqslant \inf_{\z}[P(\x,\z)+\lambda d(\z,D)].$$
Allowing $\z\in D$, we see that
$$\inf_{\y\in D} P(\x,\y)=\inf_{\z}[P(\x,\z)+\lambda d(\z,D)].\ \ \ \ \ \ \ (*)$$

\begin{proposition} \label{dfnprojection}
$D\subseteq M^n$ is definable if and only if for each definable
$P:M^{m+n}\rightarrow\Rn$, the predicate $\inf_{\y\in D}P(\x,\y)$ is definable.
\end{proposition}
\begin{proof} For the if part, take $P=d(\x,\y)$. For the converse, use the equality $(*)$
above if $P$ is $\lambda$-Lipschitz.
If $P$ is arbitrary definable, let $\phi_k^M(\x,\y)\stackrel{u}{\longrightarrow}P(\x,\y)$.
Then, verify that $\inf_{\y\in D}\phi_k^M(\x,\y)\stackrel{u}{\longrightarrow}\inf_{\y\in D}P(\x,\y)$.
\end{proof}

In particular, if $f$ is a definable function and $D$ is a definable set then $f(D)$ is definable:
$$d(x,f(D))=\inf_{t\in D} d(x,f(t)).$$

\begin{corollary} \label{restriction}
Assume $M\preccurlyeq N$,\ $D\subseteq N^m$ is definable and $d(x,D)|_M=d(x,C)$ where $C\subseteq M^m$.
Then for each definable predicate $P:N^{n+m}\rightarrow\Rn$ and $\x\in M^n$ one has that\ \
$\inf_{\y\in D}P(\x,\y)=\inf_{\y\in C}P|_M(\x,\y)$. In particular, $C$ and $D$ have the same diameter.
\end{corollary}
\begin{proof}
For the first part consider Lipschitz and non Lipschitz cases as in the proof of the preceding proposition.
Also, the diameter of $D$ is obtained by $\sup_{\x\y\in D}d(\x,\y)$.
\end{proof}

\begin{proposition} \label{dfn3conditions}
For a closed $D\subseteq M^n$ the following are equivalent:

(i) $D$ is definable.

(ii) There exists a definable predicate $P:M^n\rightarrow\Rn^+$ such that
$$\forall\x\in D,\ P(\x)=0\ \ \ \ \ \ \mbox{and}
\ \ \ \ \ \ \forall\x\in M^n, \ d(\x,D)\leqslant P(\x).$$

(iii) For each $k$ there exists a definable predicate $P_k:M^n\rightarrow\Rn$ such that
$$\forall\x\in D,\ P_k(\x)\leqslant0\ \ \ \ \ \ \mbox{and} \ \ \ \ \ \ \forall\x\in M^n, \ d(\x,D)\leqslant P_k(\x)+\frac{1}{k}.$$

\end{proposition}
\begin{proof} (i)$\Rightarrow$(ii): Take $P(\x)=d(\x,D)$.

(ii)$\Rightarrow$(iii): Take $P_k=P$.

(iii)$\rightarrow$(i): For each $k$ set $Q_k(\x)=\inf_{\y}\ [d(\x,\y)+P_k(\y)]$.
We then have that
$$Q_k(\x)\leqslant\inf_{\y\in D}\ [P_k(\y)+d(\x,\y)]\leqslant\inf_{\y\in D}d(\x,\y)=d(\x,D).$$
On the other hand, we have that $d(\x,D)\leqslant d(\x,\y)+d(\y,D)$.
So, using the assumption,
$$d(\x,D)-\frac{1}{k}\leqslant\inf_{\y}\ [d(\x,\y)+d(\y,D)-\frac{1}{k}]\leqslant Q_k(\x).$$
We conclude that $d(\x,D)$ is the uniform limit of $Q_k(\x)$, hence definable.
\end{proof}

%For example, end-sets of bi-Lipschitz definable predicates are definable.
%If $0\leqslant P$ and for some $\lambda>0$,\ \ $\lambda d(\x,\y)\leqslant|P(\x)-P(\y)|$,
%then $d(\x,Z(P))\leqslant\frac{1}{\lambda}P(\x)$ for all $\x$.

\begin{lemma} \label{zerosets}
Let $P,Q:M^n\rightarrow\Rn^+$ be definable where $M$ is extremally $\aleph_0$-saturated.
Then, $Z(P)\subseteq Z(Q)$ if and only if for each $\epsilon>0$ there is $\lambda\geqslant0$
such that for all $\x\in M$ one has that $Q(\x)\leqslant\lambda P(\x)+\epsilon$.
\end{lemma}
\begin{proof}
We prove the non-trivial part. First assume $M$ is $\aleph_0$-saturated.
Assume the claim does not hold. So, there exists $\epsilon>0$ such that
the set $$\{Q(\x)\geqslant\lambda P(\x)+\epsilon:\ \ \lambda\geqslant0\}$$
is satisfiable in $N$ by say $\c\in N$. We must therefore have that $P(\c)=0$ and hence
$Q(\c)=0$ which is impossible.

For the extremal case, let $M\preccurlyeq N$ where $N$ is $\aleph_0$-saturated.
We have only to show that $Z(P^N)\subseteq Z(Q^N)$. Assume not.
Then $P^N(\b)=0$ and $Q^N(\b)=r$ for some $r>0$ and $\b\in N$.
We may assume $r$ is the biggest real number with this property.
So, indeed $$T,P(\x)\leqslant0\vDash Q(\x)\leqslant r.$$
We conclude that $\{P(\x)\leqslant0,\ r\leqslant Q(\x)\}$ is a satisfiable facial type,
hence satisfiable in $M$. This is a contradiction.
\end{proof}

The following corollary gives a simpler condition for definability of an end-set if the model is
extremally saturated.

\begin{corollary} \label{satzeroset}
Let $M$ be extremally $\aleph_0$-saturated and $P:M^n\rightarrow\Rn^+$ be definable.
Then $D=Z(P)$ is definable if and only if for each $\epsilon>0$,
there exists $\lambda\geqslant0$ such that $$d(\x,D)\leqslant\lambda P(\x)+\epsilon\ \ \ \ \ \ \forall\x\in M.$$
\end{corollary}
\begin{proof} If $D$ is definable, the mentioned condition holds by Lemma \ref{zerosets} since
$P$ and $d(\x,D)$ have the same zeroset. Conversely, assume the above condition holds.
For each $k$ take $\lambda_k$ such that $d(\x,D)\leqslant\lambda_k P(\x)+\frac{1}{k}$.
Then, part (iii) of Proposition \ref{dfn3conditions} holds for the predicate $P_k=\lambda_k P(\x)$.
Hence, $D$ is definable.
\end{proof}

In general, zerosets have little chance to be definable.
The following proposition shows that in big models, type-definable sets are either trivial or big.

\begin{proposition} \label{dfncompact}
Let $M$ be $\aleph_0$-saturated and $D\subseteq M^n$ a nonempty compact type-definable set.
Then $D$ is a singleton.
\end{proposition}
\begin{proof} Assume $n=1$. Let $D$ be type-defined by $\Gamma(x)$ and $d(a,b)=r>0$ for $a,b\in D$.
We first show that the partial type
$$\Sigma=\{\frac{r}{2}\leqslant d(x_i,x_j)\ : \ \ \  i<j<\omega\}$$
is affinely realized in the set $\{a,b\}$. Take a condition
$$\ \ \ \ \ \ \ A_n=\frac{r}{2}\sum_{i<j\leqslant n}\alpha_{ij}\leqslant\sum_{i<j\leqslant n}\alpha_{ij} d(x_i,x_j)
=\sigma_{n}(x_0,...,x_{n})\ \ \ \ \ \ \ \ \alpha_{ij}\geqslant0$$
and assume by induction that
$A_{n-1}\leqslant\sigma_{n-1}^M(e_0,...,e_{n-1})$ where $e_i$ is either $a$ or $b$.
Verify that for one of $e_n=a$ and $e_n=b$ we must have that
$$\frac{r}{2}\sum_{i<n}\alpha_{in}\leqslant\sum_{i<n}\alpha_{in}d(e_i,e_n).$$
Hence $$A_n=A_{n-1}+\frac{r}{2}\sum_{i<n}\alpha_{in}\leqslant\sigma^M_{n-1}(e_0,...,e_{n-1})+\sum_{i<n}\alpha_{in}d(e_i,e_n)=\sigma_n^M(e_0,...,e_n).$$
We have just shown that $\Sigma\cup\Gamma(x_1)\cup\Gamma(x_2)\cup\cdots$
is affinely satisfiable in $M$.
By $\aleph_0$-saturation, any finite part of $\Sigma$ is satisfied by elements of $D$.
This contradicts the compactness of $D$.
\end{proof}

In particular, the notion of algebraic closure defined in CL or first order logic is meaningless in AL.
We can however define definable closure of a set.
The notion definable set over a set $A\subseteq M$ of parameter is defined in the usual way.
A tuple $\a\in M^n$ is said to be $A$-definable if $d(\x,\a)$ is $A$-definable.
As stated above, projection of a definable set is definable.
So, if $\a$ is $A$-definable, then every $a_i$ is $A$-definable.
Conversely, if every $a_i$ is $A$-definable, then
$d(\x,\a)=\sum_{i=1}^n d(x_i,a_i)$ which shows that $\a$ is $A$-definable.
For $A\subseteq M$,\ \ $\textrm{dcl}_M(A)$ denotes the set of points which are $A$-definable.
Clearly, it is topologically closed.

\begin{proposition} \label{well-dcl}
Let $A\subseteq M\preccurlyeq N$. Then $dcl_M(A)=dcl_N(A)$.
\end{proposition}
\begin{proof} It is sufficient to prove the claim for the case where $N$ is $\aleph_0$-saturated.
There is no harm if we further assume $A=\emptyset$. Let $a\in dcl_M(\emptyset)$.
The unique definable extension of $d(x,a)$ to $N$ satisfies the conditions (i)-(iii) of Remark \ref{3conditions}.
So, for some definable $D\subseteq N$ we have that $d(x,D)|_M=d(x,a)$.
By Corollary \ref{restriction}, $D=\{a\}$ and hence $a\in dcl_N(\emptyset)$.
Conversely assume $a\in dcl_N(\emptyset)$. Let $P(x)= d(x,a)|_M$.
Then $(M,P)\preccurlyeq (N,d(\cdot,a))$
and hence $$\inf_{x\in M} P(x)=\inf_{x\in N}d(x,a)=0.$$
For each $k$ take $a_k\in M$ such that $0\leqslant P(a_k)\leqslant\frac{1}{k}$.
Then, $d (a_k,a)=P(a_k)\leqslant\frac{1}{k}$ which means that $a_k\rightarrow a$.
Therefore, $a\in M$ and $P(a)=0$.
We have also that $$d^M(x,a)=d^N(x,a)=P(x) \ \ \ \ \ \ \ \forall x\in M$$
which shows that $a$ is definable in $M$.
\end{proof}

So, $\textrm{dcl}_M(A)$ does not depend on $M$ and we may simply denote it
by $\textrm{dcl}(A)$. The following properties are also proved easily:

\indent 1. $A\subseteq \textrm{dcl}(A)$.\\
\indent 2. If $A\subseteq\textrm{dcl}(B)$ then $\textrm{dcl}(A)\subseteq\textrm{dcl}(B)$.\\
\indent 3. If $a\in\textrm{dcl}(B)$ then $a\in\textrm{dcl}(A)$ for some countable $A\subseteq B$.\\
\indent 4. If $A$ is a dense subset of $B$ then $\textrm{dcl}(A)=\textrm{dcl}(B)$.\\
\indent 5. If $h:M^n\rightarrow M$ is $A$-definable and $\a\in\textrm{dcl}(A)$ then $h(\a)\in\textrm{dcl}(A)$.

\section{Principal types}
It is not true that if the logic and metric topologies coincide at a type $p$,
then $p$ is realized in every model.
For example, for the theory of probability algebras we have that $S_1(\textrm{PrA})=[0,1]$
and the two topologies coincide.
However, only the extreme types are realized in the model $\{0,1\}$.
For a complete type $p(\x)$ set $p(M)=\{\a\in M^n:\ tp(\a)=p\}$.

\begin{proposition} \label{principal type}
Assume $p(M)$ is nonempty definable for some $M\vDash T$.
Then $p(N)$ is nonempty definable for any $N\vDash T$ which is extremally $\aleph_0$-saturated.
\end{proposition}
\begin{proof} First, suppose that $M\preccurlyeq N$ where $N$ is $\aleph_0$-saturated.
Let $P(\x)=d(\x,p(M))$. So, $(M,P)\preccurlyeq (N,P^N)$
and $P^N$ satisfies the conditions (i)-(iii) of Remark \ref{3conditions}.
Hence $P^N(\x)=d(\x,D)$ where $D$ is the zeroset of $P^N$. We show that $D=p(N)$.

Take a condition $\phi(\x)\leqslant0$ in $p(\x)$.
For each $\a\in M$ and $\b\in p(M)$ we have that
$$\phi^M(\a)\leqslant\phi^M(\a)-\phi^M(\b)\leqslant\lambda_{\phi}\ d(\a,\b).$$
So, $$\phi^M(\a)\leqslant\lambda_{\phi}\ d(\a,p(M))
=\lambda_{\phi}\ P(\a) \ \ \ \ \ \ \ \forall \a\in M$$
and hence $$\phi^N(\a)\leqslant\lambda_{\phi}P^N(\a)\ \ \ \ \ \ \ \forall \a\in N.$$
In particular, $\phi^N(\a)\leqslant0$ for each $\a$ with $P^N(\a)=0$.
We conclude that $D\subseteq p(N)$ .
For the reverse inclusion, assume $M\vDash |\phi_k(\x)-P(\x)|\leqslant\frac{1}{k}$ for each $k$.
Since $p(M)$ is nonempty, $-\frac{1}{k}\leqslant\phi_k(\x)\leqslant\frac{1}{k}$ must belong to $p(\x)$.
Therefore, for any $k$ and $\b\in p(N)$ we have that
$$0\leqslant P^N(\b)\leqslant|P^N(\b)-\phi_k^N(\b)|+|\phi_k^N(\b)|\leqslant\frac{2}{k}.$$
This shows that $P^N(\b)=0$ for each $\b\in p(N)$.

Now assume $N$ is extremally $\aleph_0$-saturated and take an $\aleph_0$-saturated $K$ such that $M\preccurlyeq K$ and $N\preccurlyeq K$.
Then, $p(K)$ is definable. So, $p(N)=p(K)\cap N^n$ is definable by Proposition \ref{dfnrestriction}.
\end{proof}

A type $p(\x)$ is called \emph{principal} if $p(M)$ is nonempty definable for some $M\vDash T$.
Every principal type $p$ is extreme since it is exposed by $d(\x,p(M))$.
A consequence of Proposition \ref{principal type} is that if $p(\x,\y)$ is principal then so is $q(\x)=p|_{\x}$.
In fact, $q(M)$ is the projection on $p(M)$ on $M^n$ if $M$ is $\aleph_0$-saturated and $|\x|=n$.

As stated before, if $P$ is a definable predicate we may treat it as a formula in the
definitional expansion of $T$ to $L\cup\{P\}$.
Also, every type $p:\mathbb D_n\rightarrow\Rn$ has a natural extension to $\mathbf D_n$
so that $p(Q)=\lim p(\phi_k)$ whenever $\phi_k\stackrel{u}\longrightarrow Q$.
In particular, $p(\x)\vDash P(\x)\leqslant0$ if and only if for each $\a\in M\vDash T$,\
$\a\vDash p$ implies that $P(\a)\leqslant0$.
The following proposition states that $p$ is principal if and only if logic and metric topologies coincide at $p$.

\begin{proposition} \label{princip}
Let $p\in S_n(T)$. Then the following are equivalent:

(i) $p$ is principal

(ii) For each $k$ there is a definable predicate $P_k(\x)$ such that
$$T\vDash0\leqslant P_k(\x),\ \ \ \ \ \ p(\x)\vDash P_k(\x)\leqslant0\ \ \ \ \textrm{and}
\ \ \ \ \ [P_k<1]\subseteq B(p,\frac{1}{k}).$$
\end{proposition}
\begin{proof} Let $M$ be $\aleph_0$-saturated.

(i)$\Rightarrow$(ii): The requirement holds with $P_k(\x)=kd(\x,p(M))$.

(ii)$\Rightarrow$(i):
For each $\a\in M$, we have either $P_k(\a)<1$ or $1\leqslant P_k(\a)$.
In the first case we have that $d(tp(\a),p)<\frac{1}{k}$.
So, by saturation $$d(\a,p(M))<\frac{1}{k}\leqslant P_k(\a)+\frac{1}{k}.$$
In the second case, $d(\a,p(M))\leqslant1\leqslant P_k(\a)$.
So, for any $\a\in M$, $d(\a,p(M))\leqslant P_k(\a)+\frac{1}{k}$.
We conclude by part (iii) of Proposition \ref{dfn3conditions} that $p(M)$ is definable.
\end{proof}

\section{Definability in first order models}
Part (ii) of Proposition \ref{dfn3conditions} helps us to give an easy description of definable sets in case $M\vDash T$ is first order.
Let $P:M^n\rightarrow\Rn^+$ be definable and assume $$\inf\{P(\a):\ \ \a\in M,\ 0<P(\a)\}=r>0.$$
Then $$d(\x,Z(P))\leqslant\frac{1}{r}P(\x)\ \ \ \ \ \ \ \forall\x\in M.$$
Therefore, $Z(P)$ is definable in $M$.
%In particular, every formula $\phi$ with $0\leqslant\phi^M$ satisfies $(*)$ above. So, $Z(\phi^M)$ is definable.
In particular, if $C,D\subseteq M^n$ are definable, then so are $C\cap D$ and $M^n-C$ since they are zerosets of
$d(\x,C)+d(\x,D)$ and $1-d(\x,C)$ respectively. Also, if $D\subseteq M^{n+1}$ is definable,
its projection $C$ on $M^n$ is definable since $d(\x,C)=\inf_y d(\x y,D)$.
Finally, if $C,D$ are definable, then $C\times D$ is definable as the zeroset of $d(\x,C)+d(\y,D)$.
We conclude by the prenex normal form theorem in first order logic that every first order definable set in $M$
is AL-definable.
In fact, if $D\subseteq M^n$ is first order defined by say $\theta$, then
$d(\x,D)$ is the interpretation of an AL-predicate which depends only on $\theta$.
%To summarize, $D\subseteq M^n$ is first order definable if and only if it is AL-definable if and only if
%it is the zeroset of $\phi^M(\x)$ for some formula $\phi$ with $0\leqslant\phi^M$.

\begin{proposition} \label{FO-types} Assume $M, N\vDash T$ are first order. Then

(i) Every first order definable set in $M$ is AL-definable.

(ii) $\a,\b\in M$ have the same first order type iff they have the same affine type.

(iii) Every $\a\in M$ has an extreme type (and an exposed one if $L$ is countable).

(iv) If $M\preccurlyeq_{\emph{AL}}N$, then $M\preccurlyeq_{\emph{FO}}N$.

(v) $M\equiv_{\emph{FO}}N$.
\end{proposition}
\begin{proof}
(i) Explained above. (ii) The first order type of $\a\in M$ is determined by the affine conditions of the form
$d(\x,\theta(M))=0$ where $\theta$ is in the first order type of $\a$.
(iii) By Remark \ref{facial}, any satisfiable set of conditions of the form $d(\x,\theta(M))=0$ is face.
So, $tp(\a)$ is extreme. If $L$ is countable, $tp(\a)$ is exposed by the predicate $\sum_{k=0}^n 2^{-k}d(\x,\theta_k(M))$
where $\theta_k$ is an enumeration of first order formulas in $L$ satisfied by $\a$.
(iv) We may assume without loss that every $a\in M$ is named by a constant symbol $c_a\in L$.
Let $\theta(y)$ be a first order formula and assume $N\vDash\exists y\theta(y)$.
Then, $\theta(N)$ and hence $\theta(N)\cap M$ is nonempty (as the zeroset $d(y,\theta(N))$ in $M$).
We conclude by Tarski's test that $M\preccurlyeq_{\emph{FO}}N$.
(v) Let $\mathcal F$ be a countably incomplete $\kappa$-good ultrafilter where $|L|+\aleph_0+|M|<\kappa$.
An easy back and forth argument using Proposition \ref{sat1} shows that
there exists an AL-elementary embedding $f:M\rightarrow N^{\mathcal F}$.
We conclude that $M\equiv_{\emph{FO}}N^{\mathcal F}\equiv_{\emph{FO}}N$.
\end{proof}

In particular, for any first order $M$, $N$, $M\equiv_{\emph{AL}}N$ implies that $M\equiv_{\emph{FO}}N$.
A model $M\vDash T$ is called \emph{extremal} if $tp(\a)$ is extreme for each $\a\in M$.
%We mention a result from \cite{extreme}.
%\begin{theorem} \label{strong omitting}
%Every complete theory $T$ has an extremal model. Extremal models of $T$ form a CL-complete theory.\end{theorem}
It is not hard to see that if $M$ is extremal, then so is $(M,a)_{a\in A}$ for each $A\subseteq M$ (or see \cite{extreme} for the proof).
%if $tp(\a\b)$ is extreme, then so is $tp(\a/\b)$. This is also true for infinite $\a$ and $\b$.

\begin{theorem}
If $T$ has a first order model, its first order models form a complete first order theory.
These models are exactly the extremal models of $T$.
\end{theorem}
\begin{proof}
Let $M\vDash T$ be first order and $\mathcal F$ be a countably incomplete $\kappa$-good ultrafilter where $|L|+\aleph_0<\kappa$.
A back and forth argument shows that every extremal $N\vDash T$ with $|N|<\kappa$ is elementarily embedded in $M^{\mathcal F}$.
Since $\kappa$ can be taken to be arbitrarily large, every extremal model of $T$ is first order.
Therefore, regarding part (iii) of Proposition \ref{FO-types}, first order models of $T$ are exactly its extremal models.
Moreover, these models form a first order theory since they are closed under ultraproduct and first order elementary equivalence.
Also, by part (v) of Proposition \ref{FO-types}, this theory is first order complete.
\end{proof}

Let $T_e$ denote the extremal theory of $T$ and $\mathbb{S}_n(T_e)$ be the set of its first order $n$-types.
Then the restriction map $\mathbb{S}_n(T_e)\rightarrow S_n(T)$ is continuous and injective by part (ii) of Proposition \ref{FO-types}.
So its range $E_n(T)$ consisting of extreme types of $T$ is closed.
Therefore, they are homeomorphic and we may write $\mathbb{S}_n(T_e)=E_n(T)$.

We recall some notions from Choquet theory (see \cite{Alfsen} for more details).
Let $K$ be a compact convex set in a locally convex space and $\mu$ be a Baire probability measure on $K$.
The \emph{barycenter} of $\mu$ is the unique $x\in K$ such that
$$\ \ h(x)=\int h\ d\mu\ \ \ \ \ \ \ \ \ \forall h\in \mathbf{A}(K)$$
where $\mathbf{A}(K)$ is the set of affine continuous real valued functions on $K$.
In this case, one says that $\mu$ represents $x$. Every Baire probability measure on $K$ has a barycenter.
Also, by the Choquet-Bishop-de Leeuw theorem, every $x\in K$ is represented by a boundary probability measure,
i.e. a Baire probability measure $\mu$ such that $\mu(X)=0$ for every Baire set $X$ disjoint from $Ext(K)$.
If every point has a unique representation by a boundary probability measure, $K$ is called a \emph{Choquet simplex}.
If (furthermore) the boundary is closed, it is called a \emph{Bauer simplex}.
It is well known that every Baire probability measure on a compact Hausdorff space $K$ is regular
and it uniquely extends to regular Borel probability measure on $K$.
If $K=S_n(T)$, the extreme boundary is closed and we may replace Baire measures with the corresponding regular Borel extension.
Let $p_\mu$ be the barycenter of the regular Borel probability measure $\mu$ on $E_n(T)$.

\begin{theorem} \label{Bauer}
If $T$ has a first order model, then $S_n(T)$ is a Bauer simplex.
\end{theorem}
\begin{proof}
We use the equality $\mathbb{S}_n(T_e)=E_n(T)$. We have only to show that $\mu\mapsto p_\mu$ is injective.
Assume $p$ has two representations by (regular) boundary measures $\mu$ and $\nu$, i.e.
$$h(p)=\int h\ d\mu=\int h\ d\nu\ \ \ \ \ \ \ \ \ \forall h\in\mathbf{A}(S_n(T)).$$
For first order definable $D\subseteq M^n$ let $Q(\x)=1-d(\x,D)$ and $X=\{q\in E_n(T):D\in q\}$.
Then, one has that $$\mu(X)=\int\hat Q\ d\mu=\int\hat Q\ d\nu=\nu(X)$$
where $\hat Q(q)=q(Q)$ for each type $q$. So, $\mu$ and $\nu$ coincide on clopen subsets of $E_n(T)$.
Since $\mu$ and $\nu$ are regular, we have that $\mu=\nu$.
\end{proof}

Let $\mathbb T$ be a complete first order theory.
A Keisler measure for $\mathbb T$ is a finitely additive probability measure on the Boolean algebra of
first order definable subsets $M^n$ where $M\vDash\mathbb{T}$.
Let $T$ be the affine part of $\mathbb T$, i.e. the reduction of $\mathbb T$ to AL-conditions.
Then, $T_e=\mathbb T$ and every Keisler measure corresponds to a unique regular Borel probability measure on $\mathbb{S}_n(\mathbb{T})=E_n(T)$.
So, by the above proposition, affine types are uniquely represented by integration over Keisler measures.
In particular, there is a one-to-one correspondence between Keisler measures of $\mathbb T$ and affine types of $T$.
One can say that Keisler measures of $\mathbb T$ (as generalized first order types) are realized in the models its affine part.
Also, a concrete representation of Keisler measures can be given in the case $\mathbb{S}_n(\mathbb T)$ is countable.
Assume $\{p_i: i\in I\}$ is an enumeration (without repetition) of $n$-types of $\mathbb T$ where $I\subseteq\Nn$.
Then, for $M\vDash T_e$, every Keisler measure on $M^n$ has a unique representation of the form $\sum_{i\in I}r_i\mu_i$
where $r_i\geqslant0$, $\sum_{i\in I}r_i=1$ and $\mu_i$ is the Dirac measure at $p_i$.

\section{Definability in compact models} \label{section4}
In this section we assume $T$ has a compact model.
In the framework of CL, if $M$ is compact, zerosets of $\emptyset$-definable predicates are definable and
if the language is countable, they are the only type-definable sets.
The situation is different in AL.
In a compact model, a type-definable set need not be a zeroset and a zeroset need not be definable.
Moreover, definable sets are exactly the end-sets of definable predicates.
By an \emph{end-set} we mean a set of the form $\{\x:\ P(\x)=r\}$
where $r$ is either $\inf_{\x}P(\x)$ or $\sup_{\x}P(\x)$.

\begin{theorem} \label{minimal sets}
Let $M\vDash T$ be extremally $\aleph_0$-saturated.
Then a nonempty $D\subseteq M^n$ is definable if and only if there is a definable
$P:M^n\rightarrow\Rn^+$ such that $D=Z(P)$.
\end{theorem}
\begin{proof} We prove the non-trivial part.
Let $P$, $D$ be as above and consider the case $n=1$.
First assume $M$ is compact and $\epsilon>0$ is fixed.
Then there must exist $\lambda\geqslant0$ such that
$d(x,D)\leqslant\lambda P(x)+\epsilon$ for all $x\in M$.
Otherwise, for each $\lambda\geqslant0$, the set
$$X_\lambda=\{x\ : \ d(x,D)\geqslant\lambda P(x)+\epsilon\}$$
is nonempty closed. Since $M$ is compact,
there exists $b\in\cap_\lambda X_\lambda$. Clearly then $P(b)=0$ and hence
$b\in D$ and $d(b,D)\geqslant\epsilon$. This is a contradiction.
Therefore, by Corollary \ref{satzeroset}, $D$ is definable.

Now, assume $M$ is $\aleph_0$-saturated.
An easy back and forth argument shows that every compact model of $T$ can be elementarily embedded in $M$.
Let $K\preccurlyeq M$ where $K$ is compact. We then have that $(K,P^K)\preccurlyeq(M,P)$.
Moreover, $D_0=Z(P^K)$ is nonempty definable in $K$.
By Propositions \ref{dfnconditions} and \ref{definitional expansion},
for some definable $D_1\subseteq M$ we have that
$$(K,P^K,d(x,D_0))\preccurlyeq(M,P,d(x,D_1)).$$
Since $P^K$ and $d(x,D_0)$ have the same zeroset, by Lemma \ref{zerosets},
$P$ and $d(x,D_1)$ must have the same zeroset. We conclude that $D=D_1$.

Finally, assume $M$ is just extremally $\aleph_0$-saturated. Let $M\preccurlyeq N$ where $N$
is $\aleph_0$-saturated. Then $Z(P^N)$ is definable. Hence $Z(P)=Z(P^N)\cap M^n$
is definable by Proposition \ref{dfnrestriction}.
\end{proof}

It is proved in \cite{extreme} that a theory having a compact model has a unique compact extremal model.
Such a model is elementarily embedded in every extremally $\aleph_0$-saturated model of the theory.
Proof of the above theorem could be then shortened a bit by using this fact.
A consequence of Theorem \ref{minimal sets} is that in an extremally $\aleph_0$-saturated model $M$
if $f:M^n\rightarrow M^m$ and $D\subseteq M^m$ are definable then $f^{-1}(D)$ definable.
It is the zeroset of $d(f(\x),D)$.
Also, writing $P(x)=\sum_k 2^{-k}d(x,D_k)$ one checks that countable intersections of definable sets are definable.
Definable sets are not closed under finite unions. We have however the following.

\begin{proposition} \label{union}
Let $D_1\subseteq D_2\subseteq\ldots$ be a chain of definable sets in $M$.
If $M$ contains a compact elementary submodel then $D=\overline{\cup_n D_n}$ is definable.
\end{proposition}
\begin{proof} It is clear that $d(\x,D_n)$ converges to $d(\x,D)$ pointwise.
Assume $K\preccurlyeq M$ where $K$ is compact. Let $P_n(\x)=d(\x,D_n)|_K$ and $P(\x)=d(\x,D)|_K$.
Then $P_n$ is monotone and converges to $P$ pointwise.
Since $P$ is continuous, the convergence is uniform (by Dini's theorem). For each $\epsilon>0$ take $\ell$
such that
$$|P_m(\x)-P_n(\x)|\leqslant\epsilon\ \ \ \ \ \ \forall m,n\geqslant \ell,\ \ \forall \x\in K.$$
Then we must similarly have that
$$|d(\x,D_m)-d(\x,D_n)|\leqslant\epsilon\ \ \ \ \ \ \forall m,n\geqslant \ell,\ \ \forall \x\in M.$$
This shows that the convergence of $d(\x,D_n)$ to $d(\x,D)$ is uniform.
\end{proof}

As in the case of formulas, $\hat Q(p)=p(Q)$ is an affine logic-continuous function on $S_n(T)$.
A partial type $\Sigma(\x)$ is exposed if the set $$[\Sigma]=\{p\in S_n(T):\ \Sigma\subseteq p\}$$
is a face exposed by $\hat Q$ for some definable predicate $Q$.

\begin{proposition} \label{exposed types}
Let $M$ be $\aleph_0$-saturated and $\Sigma(\x)$ be a partial type. Then the following are equivalent:

(i) $\Sigma(M)=\{\a: \ \a\vDash\Sigma(\x)\}$ is definable

(ii) $[\Sigma]$ is either $S_n(T)$ or an exposed face

(iii) There exists a definable predicate $Q(\x)$ such that
$$T\vDash0\leqslant Q(\x) \ \ \ \ \ \ \& \ \ \ \ \
\Sigma\equiv\{Q(\x)=0\}.$$
\end{proposition}
\begin{proof} (i)$\Rightarrow$(ii) Let $Q(\x)=d(\x,\Sigma(M))$.
We show that for each $p$,\ \ $\Sigma\subseteq p$ if and only if $p(Q(\x))=0$.
Fix $\a\vDash p$. If $\Sigma\subseteq p$ then $p(Q)=Q(\a)=0$.
Conversely, if $p(Q)=0$, we have that $Q(\a)=p(Q)=0$ and hence $\a\vDash\Sigma$.
This implies (by saturation of $M$) that $\Sigma\subseteq p$.
Now, we have that $[\Sigma]=\{p\in S_n(T):\ \hat Q(p)=0\}$.

(ii)$\Rightarrow$(iii): If $[\Sigma]=S_n(T)$, the required conditions hold with $Q=0$.
Otherwise, there exists $Q(\x)\in\mathbf D_n$ such that $\hat Q$ is nonnegative nonconstant on $S_n(T)$
and $$[\Sigma]=\{p\in S_n(T)\ :\ \hat{Q}(p)=0\}.$$
In this case, the required conditions hold with $Q(\x)$.

(iii)$\Rightarrow$(i): The assumption implies that $\Sigma(M)=Z(Q)$. By Theorem \ref{minimal sets}, this is a definable set.
\end{proof}

As a consequence, complete principal types are exactly the exposed ones.
Also, regarding Theorem \ref{minimal sets}, there is an order preserving correspondence
between nonempty definable sets in $M^n$ and exposed faces of $S_n(T)$
so that bigger sets correspond to bigger exposed faces.

%If $\mathbb D_n$ is finite dimensional, by the Straszewicz theorem (\cite{Aliprantis-Inf} Th 7.89),
%exposed types are dense in the set of extreme types. So, principal types are dense in the extreme ones.

\section{Definability in measure algebras}
In this section we characterize one-dimensional definable sets in the theory of probability algebras.
Before this, we give some examples of definable sets in compact structures.
In a metric group, closure of the torsion subgroup is definable.
In a dynamical system $(M,f)$, the closure of the set of periodic points is definable.
In the closed unit disc with the Euclidean metric, the boundary as well as the center is definable.
Also, the line segment between two point is definable with parameters.
Some interesting points in acute triangles are definable.
For example, the circumcenter and the centroid are definable.
\bigskip

\noindent{\bf Probability algebras}: Let $L=\{\wedge,\vee,\ ' ,0,1,\mu\}$.
The theory of probability algebras PrA is axiomatized as follows:

- Axioms of Boolean algebras

- $\mu(0)=0$\  and \ $\mu(1)=1$

- $\mu(x)\leqslant\mu(x\vee y)$

- $\mu(x\wedge y)+\mu(x\vee y)=\mu(x)+\mu(y)$

- $d(x,y)=\mu(x\triangle y)$.
\bigskip

Since we assume models are metrically complete, $\mu$ is in fact sigma-additive
and the ambient Boolean algebras are Dedekind complete.
PrA is a complete $\aleph_0$-stable theory with quantifier-elimination (see \cite{bagheri3}).
For each $A\subseteq M\vDash$ PrA, let $\bar A$ be the topological closure of
the probability algebra generated by $A$. Then, $\bar A$ is a model of PrA.
We conclude by Proposition \ref{well-dcl} that $\textrm{dcl}(A)=\bar A$.
There is also an easy description of parametrically definable subsets of $M$.
We recall some definitions from \cite{Fremlin}. A bounded functions
$f:M\rightarrow\Rn$ is additive if $f(x\vee y)=f(x)+f(y)$
whenever $x\wedge y=0$. Countable additivity is defined similarly.
$f$ is said to be positive on $a$ if $0\leqslant f(t)$ for each $t\leqslant a$.
It is negative on $a$ if $-f$ is positive on $a$. For each $a$, the function
$\mu(x\wedge a)$ is countably additive. Additive functions form a vector space.
By inclusion-exclusion principle, the formula $\mu(t(\x))$ is equivalent to a finite sum of
formulas of the form $\mu(z_1\wedge\ldots\wedge z_n)$ where $z_i$ is either $x_i$ or $x_i'$.
Therefore, for each quantifier-free formula $\phi(x,\a)$, the function $\phi^M(x)-\phi^M(0)$ is countably additive.
For $a,b\in M$ let $[a,b]=\{x\in M: a\leqslant x\leqslant b\}$.

\begin{proposition} \label{PrA-def}
A closed $D\subseteq M\vDash$ PrA is definable with parameters if and only if $D=[a,b]$ for some $a,b$.
\end{proposition}
\begin{proof} If $a\leqslant b$, we have that $d(x,[a,b])=\mu(x\wedge b')+\mu(a\wedge x')$
(the minimum distance is obtained at $a\vee(b\wedge  x)$).
So, $[a,b]$ is definable for every $a,b$. Conversely assume $D\subseteq M$
is nonempty and definable.
First assume $M$ is $\aleph_0$-saturated. Suppose that $D$ is the maximum-set
(points at which $P$ takes its maximum) of a definable function $P:M\rightarrow\Rn$.
Let $\phi_k^M\rightarrow P$ uniformly.
Then $\phi_k^M(x)-\phi_k^M(0)$ tends to $f(x)=P(x)-P(0)$ and hence $f(x)$ is finitely additive.
In fact, since $f$ is continuous, it is countably additive (see \cite{Fremlin} 327B).
$D$ is the maximum-set of $f$ too. We must determine $D$.
By the Hahn decomposition theorem (\cite{Fremlin} 326I), there exists $a$
such that $f$ is positive on $a$ and negative on $a'$.
By completeness of $M$, we may further assume that $a$ is maximal with this property.
Also, there is a maximal $b$ such that $f$ is negative on $b$ and positive on $b'$.
So, $b'\leqslant a$ and $f(t)=0$ for every $t\leqslant a\wedge b$.
Moreover, by maximality of $a$ and $b$,\ \ $f(t\wedge a')<0$ whenever $t\wedge a'>0$ and
$f(t\wedge b')>0$ whenever $t\wedge b'>0$. Now, by additivity of $f$, for each $t$ we have that
$$f(t)=f(t\wedge a')+f(t\wedge a\wedge b)+f(t\wedge b')=f(t\wedge a')+f(t\wedge b').$$
We conclude that, $f$ takes its maximum value at $t$ if and only if
$b'\leqslant t\leqslant a$.

Now, assume $M$ is arbitrary. Let $M\preccurlyeq N$ be $\aleph_0$-saturated
and $Q(x)$ be the definable extension of $d(x,D)$ to $N$.
By Proposition \ref{dfnconditions}, $Q(x)=d(x,\bar D)$ where $\bar D=Z(Q)$.
Let $\bar D=[a,b]$ where $a,b\in N$. We show that $M\cap[a,b]$ is an interval in $M$. Since $M$ is Dedekind complete,
$a_1=\inf\{t\in M:\ a\leqslant t\}$ and $b_1=\sup\{t\in M:\ t\leqslant b\}$ belong to $M$.
Clearly, then $D=M\cap[a,b]=[a_1,b_1]$.
\end{proof}

Since $\{0,1\}$ is a first order model of PrA, by Theorem \ref{Bauer}, $S_n(PrA)$ is a simplex.
In fact, it is the standard $(2^n-1)$-simplex. We finish the paper by asking two question.
\bigskip

\noindent{\bf Question}: Is it true that if $T$ has a compact model, it has a compact prime model?
Is it true that if $T$ has a first order model and $T_e$ has a prime model, then its prime model is a prime model for $T$?


\begin{thebibliography}{5}
\bibitem{Aliprantis-Inf} C.D. Aliprantis, K.C. Border, \emph{Infinite dimensional analysis}, Springer (2006).
\bibitem{Alfsen} E.M. Alfsen, \emph{Compact convex sets and boundary integrals}, Springer-Verlag (1971).
\bibitem{bagheri3} S.M. Bagheri, \emph{Linear model theory for Lipschitz structures}, Arch. Math. Logic, 53: 897-927 (2014).
\bibitem{Isomorphism} S.M. Bagheri, \emph{The isomorphism theorem for linear fragments of continuous logic},
Math. Log. Quart. Volume 67, Issue 2, 193-205 (2021).
\bibitem{extreme} S.M. Bagheri, \emph{Extreme types and extremal models}, to appear in Annals of Pure and Apllied Logic.
\bibitem{BBHU} I. Ben-Yaacov, A. Berenstein, C.W. Henson, A. Usvyatsov,
\emph{Model theory for metric structures}, Model theory with Applications to Algebra and Analysis,
volume 2 (Zo\'{e} Chatzidakis, Dugald Macpherson, Anand Pillay, and Alex Wilkie, eds.),
London Math Society Lecture Note Series, vol. 350, Cambridge University Press, pp. 315-427 (2008).
\bibitem{CK1} C.C. Chang and H.J. Keisler, \textit{Model theory }, North-Holland (1990).
\bibitem{Conway} J.B. Conway, \emph{A course in functional analysis} (second edition), Springer (1990).
\bibitem{Fremlin} D.H. Fremlin, \emph{Measure theory, volume 3: Measure algebras}, Internet file (2004).
\bibitem{Rao} K.P.S. Bhaskara Rao, M. Bhaskara Rao, \textit{Theory of charges}, Academic Press (1983).
\bibitem{Simon} B. Simon, \emph{Convexity: An analytic viewpoint}, Cambridge University Press (2011).
\end{thebibliography}
\end{document}